\documentclass[12pt,reqno]{amsart}

\usepackage{hyperref}

\newtheorem{theorem}{Theorem}[section]
\newtheorem{lemma}[theorem]{Lemma}
\newtheorem{corollary}[theorem]{Corollary}

\numberwithin{equation}{section}

\begin{document}

\title[Quillen connection and uniformization of Riemann surfaces]{Quillen connection and the
uniformization of Riemann surfaces}

\author[I. Biswas]{Indranil Biswas}

\address{School of Mathematics, Tata Institute of Fundamental
Research, Homi Bhabha Road, Mumbai 400005, India}

\email{indranil@math.tifr.res.in}

\author[F. F. Favale]{Filippo Francesco Favale}

\address{Dipartimento di Matematica e Applicazioni, Universit\`a degli
Studi di Milano-Bicocca, Via Roberto Cozzi, 55, I-20125 Milano, Italy}

\email{filippo.favale@unimib.it}

\author[G. P. Pirola]{Gian Pietro Pirola}

\address{Dipartimento di Matematica, Universit\`a di Pavia,
via Ferrata 5, I-27100 Pavia, Italy}

\email{gianpietro.pirola@unipv.it}

\author[S. Torelli]{Sara Torelli}

\address{Institut f\"ur Algebraische Geometrie, Leibniz Universit\"at Hannover,
Welfengarten 1, 30167 Hannover, Germany}

\email{torelli@math.uni-hannover.de}

\subjclass[2010]{30F10, 14H15, 53C07, 53B10}

\keywords{Uniformization, projective structure, Quillen connection, torsor}

\begin{abstract}
The Quillen connection on ${\mathcal L}\, \longrightarrow\, {\mathcal M}_g$, where ${\mathcal L}^*$ is the Hodge 
line bundle over the moduli stack of smooth complex projective curves curves ${\mathcal M}_g$, $g\, \geq\, 5$, is 
uniquely determined by the condition that its curvature is the Weil--Petersson form on ${\mathcal M}_g$.
The bundle of holomorphic connections on ${\mathcal L}$ has a unique holomorphic isomorphism with the bundle on 
${\mathcal M}_g$ given by the moduli stack of projective structures. This isomorphism takes the $C^\infty$ section 
of the first bundle given by the Quillen connection on ${\mathcal L}$ to the $C^\infty$ section of the second 
bundle given by the uniformization theorem. Therefore, any one of these two sections determines the other uniquely.
\end{abstract}

\maketitle

\section{Introduction}\label{se1}

A holomorphic $\Omega^1_Z$--torsor over a complex manifold $Z$ is a holomorphic fiber bundle ${\mathcal E}$ over $Z$
on which the holomorphic cotangent bundle $\Omega^1_Z$ acts satisfying the condition that the map
$$
{\mathcal E}\times_Z \Omega^1_Z\, \longrightarrow\, {\mathcal E}\times_Z {\mathcal E},
$$
constructed using the action map ${\mathcal E}\times_Z \Omega^1_Z\, \longrightarrow\, {\mathcal E}$ and the identity map 
of ${\mathcal E}$, is a biholomorphism. This notion extends to smooth Deligne--Mumford stacks.

Here we investigate two natural holomorphic $\Omega^1_{{\mathcal M}_g}$--torsors on the 
moduli stack ${\mathcal M}_g$ of smooth complex projective curves of genus $g$ (throughout we assume that $g\, 
\geq\, 5$). The first $\Omega^1_{{\mathcal M}_g}$--torsor is given by the moduli stack ${\mathcal P}_g$ of genus 
$g$ surfaces equipped with a projective structure. We recall that the projective structures on a Riemann surface 
$X$ is an affine space modeled on $H^0(X,\, K^{\otimes 2}_X)
\,=\, (\Omega^1_{{\mathcal M}_g})_X$. The second $\Omega^1_{{\mathcal M}_g}$--torsor is given by the sheaf of 
holomorphic connections ${\rm Conn}({\mathcal L})$ on the dual $\mathcal L$ of the Hodge line bundle on ${\mathcal 
M}_g$. We recall that the space of holomorphic connections on ${\mathcal L}\big\vert_U$, where $U\, \subset
\, {\mathcal M}_g$ is an affine open subset, is an affine space modeled on the vector space $H^0(U,\, \Omega^1_U)$.

The uniformization theorem gives a $C^\infty$ section
$$
\Psi\,\, :\,\, {\mathcal M}_g \, \longrightarrow\, {\mathcal P}_g\, .
$$
On the other hand, the holomorphic line bundle ${\mathcal L}$ has a complex connection associated to
the Quillen metric on it. This Quillen connection is uniquely determined, among all complex connections on
${\mathcal L}$, by the property that its curvature is
$$
\frac{\sqrt{-1}}{6\pi}\omega_{WP}\, ,
$$
where $\omega_{WP}$ is the Weil--Petersson form on ${\mathcal M}_g$ (see Corollary \ref{cor1}).

We construct from ${\rm Conn}({\mathcal L})$ a new holomorphic $\Omega^1_{{\mathcal M}_g}$--torsor
simply by scaling the action of $\Omega^1_{{\mathcal M}_g}$. More precisely, if
$$
A\,\, :\,\, {\rm Conn}({\mathcal L})\times_{{\mathcal M}_g} \Omega^1_{{\mathcal M}_g}\, \longrightarrow\,
{\rm Conn}({\mathcal L})
$$
is the action of $\Omega^1_{{\mathcal M}_g}$ on ${\rm Conn}({\mathcal L})$, then define the following new
holomorphic action of $\Omega^1_{{\mathcal M}_g}$ on the same holomorphic fiber bundle ${\rm Conn}({\mathcal L})$:
$$
A^t\,\, :\,\, {\rm Conn}({\mathcal L})\times_{{\mathcal M}_g} \Omega^1_{{\mathcal M}_g}\, \longrightarrow\,
{\rm Conn}({\mathcal L})\, ,\ \ (z,\, v)\, \longmapsto\, A\left(z,\, \frac{\sqrt{-1}}{6\pi}\cdot v\right)\, .
$$
The resulting holomorphic $\Omega^1_{{\mathcal M}_g}$--torsor
$({\rm Conn}({\mathcal L}),\, A^t)$ will be denoted by ${\rm Conn}^t({\mathcal L})$.

Let
$$
\Phi\,\, :\,\, {\mathcal M}_g \, \longrightarrow\, {\rm Conn}^t({\mathcal L})\,=\, {\rm Conn}({\mathcal L})
$$
be the $C^\infty$ section given by the Quillen connection on ${\mathcal L}$.

We prove the following (see Theorem \ref{thm1}):

\begin{theorem}\label{thmi}\mbox{}
\begin{enumerate}
\item There is exactly one holomorphic isomorphism between the $\Omega^1_{{\mathcal M}_g}$--torsors
${\rm Conn}^t({\mathcal L})$ and ${\mathcal P}_g$.

\item The holomorphic isomorphism between the $\Omega^1_{{\mathcal M}_g}$--torsors
${\rm Conn}^t({\mathcal L})$ and ${\mathcal P}_g$ takes the above section $\Phi$ of ${\rm Conn}^t({\mathcal L})$
to the section $\Psi$ of ${\mathcal P}_g$ given by the uniformization theorem.
\end{enumerate}
\end{theorem}

We note that from Theorem \ref{thmi} it follows that each of the sections $\Phi$
and $\Psi$ determines the other uniquely.

\section{Quillen metric on a line bundle}

For $g\, \geq\, 5$, let ${\mathcal M}_g$ denote the moduli stack of smooth complex projective curves
of genus $g$. It is an irreducible smooth complex quasiprojective orbifold of dimension $3(g-1)$
\cite{DM}. Moreover, ${\mathcal M}_g$ has a natural K\"ahler structure
\begin{equation}\label{e0}
\omega_{WP}\,\, \in\,\, C^\infty({\mathcal M}_g,\, \Omega^{1,1}_{{\mathcal M}_g})
\end{equation}
which is known as the {\it Weil--Petersson form}.

A torsor over ${\mathcal M}_g$ for the holomorphic cotangent bundle
$\Omega^1_{{\mathcal M}_g}$ is a fiber bundle $${\mathcal E}\, \longrightarrow\, {\mathcal M}_g$$ together
with a morphism
$$
\varpi\,\, :\,\, {\mathcal E}\times_{{\mathcal M}_g} \Omega^1_{{\mathcal M}_g}\, \longrightarrow\, {\mathcal E}
$$
such that
\begin{itemize}
\item $\varpi_X$ is an action of the vector space $(\Omega^1_{{\mathcal M}_g})_X$ on the fiber
${\mathcal E}_X$ for every $X\, \in\, {\mathcal M}_g$, and

\item the map of fiber products
$$
{\mathcal E}\times_{{\mathcal M}_g} \Omega^1_{{\mathcal M}_g}\, \longrightarrow\, {\mathcal E}
\times_{{\mathcal M}_g}{\mathcal E},\, \ \ (e,\, v)\, \longmapsto\, (e,\, \varpi(e,\, v))
$$
is an isomorphism.
\end{itemize}

Let
\begin{equation}\label{e1}
\varphi\, :\, {\mathcal C}_g\, \longrightarrow\, {\mathcal M}_g
\end{equation}
be the universal curve. The line bundle on ${\mathcal M}_g$
$$
{\mathcal L}\, :=\, \det R^1\varphi_*{\mathcal O}_{{\mathcal C}_g}\,=\,
\bigwedge\nolimits^g R^1 \varphi_*{\mathcal O}_{{\mathcal C}_g}
$$
is a generator of ${\rm Pic}({\mathcal M}_g)\,=\, \mathbb Z$ \cite[p.~154, Theorem 1]{AC}.

Let ${\rm Conn}({\mathcal L})\, \longrightarrow\, {\mathcal M}_g$ be the holomorphic fiber bundle given by the
sheaf of holomorphic connections on ${\mathcal L}$. We will briefly recall the construction of
${\rm Conn}({\mathcal L})$. Consider the Atiyah exact sequence
\begin{equation}\label{f1}
0\, \longrightarrow\, {\mathcal O}_{{\mathcal M}_g}\,=\, {\rm Diff}^0({\mathcal L},\, {\mathcal L})
\, \longrightarrow\, {\rm At}({\mathcal L})\, :=\, {\rm Diff}^1({\mathcal L},\, {\mathcal L})
\, \stackrel{p_0}{\longrightarrow}\, T{\mathcal M}_g \, \longrightarrow\, 0\, ,
\end{equation}
where ${\rm Diff}^i({\mathcal L},\, {\mathcal L})$ is the holomorphic vector bundle over ${\mathcal M}_g$
corresponding to the sheaf of holomorphic differential operators of order $i$ from ${\mathcal L}$ to itself,
$T{\mathcal M}_g$ is the holomorphic tangent bundle of ${\mathcal M}_g$, and $p_0$ is the symbol map. For
any open subset $U\, \subset\,
{\mathcal M}_g$, giving a holomorphic connection on ${\mathcal L}\big\vert_U$ is equivalent to giving
a holomorphic splitting of \eqref{f1} over $U$ \cite{At}. Let
\begin{equation}\label{e3}
0\, \longrightarrow\, \Omega^1_{{\mathcal M}_g} \, \longrightarrow\,{\rm At}({\mathcal L})^*
\, \stackrel{\beta}{\longrightarrow}\, {\mathcal O}_{{\mathcal M}_g} \, \longrightarrow\, 0
\end{equation}
be the dual of the sequence in \eqref{f1}. Let ${\mathbf 1}_{{\mathcal M}_g}\, :\, {\mathcal M}_g\, \longrightarrow\,
{\mathcal O}_{{\mathcal M}_g}$ denote the section given by the constant function $1$ on ${\mathcal M}_g$.
Then
\begin{equation}\label{e4}
{\rm At}({\mathcal L})^*\, \supset\, \beta^{-1}({\mathbf 1}_{{\mathcal M}_g}({\mathcal M}_g))\, =:\,
{\rm Conn}({\mathcal L})\, \stackrel{\phi}{\longrightarrow}\, {\mathcal M}_g
\end{equation}
where $\beta$ is the projection in \eqref{e3}. From \eqref{e3} it follows immediately that
${\rm Conn}({\mathcal L})$ is a holomorphic torsor over ${\mathcal M}_g$ for the holomorphic cotangent bundle
$\Omega^1_{{\mathcal M}_g}$. In particular, for any $X\, \in\, {\mathcal M}_g$ the complex vector
space $(\Omega^1_{{\mathcal M}_g})_X$ acts freely transitively on the fiber of ${\rm Conn}({\mathcal L})$
over $X$. Let
\begin{equation}\label{eA}
A\,\, :\,\, {\rm Conn}({\mathcal L})\times_{{\mathcal M}_g} \Omega^1_{{\mathcal M}_g}\, \longrightarrow\,
{\rm Conn}({\mathcal L})
\end{equation}
be the holomorphic map giving the torsor structure. For any $C^\infty$ $(1,\, 0)$-form $\eta$ on ${\mathcal M}_g$, let
\begin{equation}\label{eA2}
A_\eta\,\, :\,\, {\rm Conn}({\mathcal L})\, \longrightarrow\, {\rm Conn}({\mathcal L}),\, \ \
z\, \longmapsto\, A(z,\, \eta(\phi(z)))
\end{equation}
be the $C^\infty$ automorphism of ${\rm Conn}({\mathcal L})$ over ${\mathcal M}_g$, where $\phi$ and $A$ are
the maps in \eqref{e4} and \eqref{eA} respectively.

A {\it complex connection} on ${\mathcal L}$ is a $C^\infty$ connection $\nabla$ on ${\mathcal L}$ such that
the $(0,\,1)$-component $\nabla^{0,1}$ of $\nabla$ coincides with the Dolbeault operator on $\mathcal L$ that defines
the holomorphic structure of $\mathcal L$. The space of complex connections on $\mathcal L$ is in a
natural bijection with the space of $C^\infty$ sections ${\mathcal M}_g\, \longrightarrow\,
{\rm Conn}({\mathcal L})$ of the projection $\phi$ in \eqref{e4}. There is a tautological holomorphic connection
$D^0$ on the line bundle $\phi^*{\mathcal L}$, whose curvature
$\Theta\,=\, {\rm Curv}(D^0)$ is a holomorphic symplectic form on ${\rm Conn}({\mathcal L})$
(see \cite[p.~372, Proposition 3.3]{BHS}).
Any complex connection $\nabla$ on ${\mathcal L}$ satisfies the
equation $$\nabla\,=\, f^*_{\nabla} D^0\, ,$$ where $f_{\nabla}\, :\, {\mathcal M}_g\, \longrightarrow\,
{\rm Conn}({\mathcal L})$ is the $C^\infty$ section corresponding to
$\nabla$. Consequently, the curvature ${\rm Curv}(\nabla)$ of $\nabla$ satisfies the equation
\begin{equation}\label{eA4}
{\rm Curv}(\nabla)\,=\, f^*_{\nabla} \Theta\, .
\end{equation}
We also have
\begin{equation}\label{eA3}
A^*_\eta\Theta\,=\, \Theta+d\eta,
\end{equation}
where $A_\eta$ is the map in \eqref{eA2}. 

Given a Hermitian structure $h_1$ on $\mathcal L$,
there is a unique complex connection on $\mathcal L$ that preserves $h_1$ \cite[p.~11, Proposition 4.9]{Ko};
it is known as the {\it Chern connection}.

Equip the family of Riemann surfaces
${\mathcal C}_g$ in \eqref{e1} with the relative Poincar\'e metric. Also, equip ${\mathcal O}_{{\mathcal
C}_g}$ with the trivial (constant) Hermitian structure; the pointwise norm of the
constant section with value $c$ is the constant $|c|$. These two
together produce a Hermitian structure $h_Q$ on
$\mathcal L$, which is known as the Quillen metric \cite{Qu}, \cite{BGS}. Let
$$
{\rm Curv}(\nabla^Q)\, \in\, C^\infty({\mathcal M}_g,\, \Omega^{1,1}_{{\mathcal M}_g})
$$
be the curvature of the Chern connection $\nabla^Q$ on ${\mathcal L}$ for the Hermitian structure $h_Q$;
this $\nabla^Q$ is known as the {\it Quillen connection}. Then
\begin{equation}\label{e5}
{\rm Curv}(\nabla^Q)\,=\, \frac{\sqrt{-1}}{6\pi}\omega_{WP}\, ,
\end{equation}
where $\omega_{WP}$ is the K\"ahler form in \eqref{e0} \cite[p.~184, Theorem 2]{ZT1}; a much
more general result is proved in \cite[p.~51, Theorem 0.1]{BGS} from which \eqref{e5}
follows immediately. Let
\begin{equation}\label{e6}
\Phi\,\, :\,\, {\mathcal M}_g \, \longrightarrow\, {\rm Conn}({\mathcal L})
\end{equation}
be the $C^\infty$ section of the projection $\phi$ in \eqref{e4} given by 
the above Quillen connection $\nabla^Q$.

\begin{lemma}\label{lem1}
There is exactly one complex connection $\nabla$ on ${\mathcal L}$ such that
the curvature ${\rm Curv}(\nabla)$ of $\nabla$ satisfies the equation
$$
{\rm Curv}(\nabla)\,=\, \,=\, \frac{\sqrt{-1}}{6\pi}\omega_{WP}\,.
$$
\end{lemma}

\begin{proof}
{}From \eqref{e5} we know that the Quillen connection $\nabla^Q$ satisfies this equation.
Let $\nabla$ be another connection on ${\mathcal L}$ satisfying this condition. Consider the
$C^\infty$ $(1,\, 0)$-form $\eta_0\, =\, \nabla^Q-\nabla$ on ${\mathcal M}_g$. From
\eqref{eA3} and \eqref{eA4} it follows that $d\eta_0\,=\, {\rm Curv}(\nabla^Q)-{\rm Curv}(\nabla)\,=\, 0$.

It is known that ${\mathcal M}_g$ does not admit any nonzero closed $(1,\, 0)$-form
(see \cite[p.~228, Theorem 2]{Mu}, \cite[Lemma 1.1]{Ha}). In fact, ${\mathcal M}_g$ does not admit
any nonzero holomorphic $1$-form \cite[Theorem 3.1]{FPT}; recall that $g\, \geq\, 5$. So
we have $\eta_0\,=\,0$, and hence $\nabla^Q\,=\,\nabla$.
\end{proof}

Lemma \ref{lem1} has the following immediate consequence:

\begin{corollary}\label{cor1}
The curvature equation \eqref{e5} uniquely determines the Quillen connection $\nabla^Q$ among the space of all
complex connections on $\mathcal L$.
\end{corollary}

\subsection{Projective structures and uniformization}

Take any smooth complex projective curve $X$. A holomorphic coordinate chart on $X$ is a pair
of the form $(U,\, f)$, where $U\, \subset\, X$ is an analytic open subset and $f\,
:\, U\, \longrightarrow\, {\mathbb C}{\mathbb P}^1$ is a holomorphic embedding. A holomorphic
coordinate atlas on $X$ is a collection of coordinate charts $\{(U_i,\, f_i)\}_{i\in I}$
such that $$X\,=\, \bigcup_{i\in I} U_i\, .$$ A projective structure on $X$ is given by a holomorphic
coordinate atlas $\{(U_i,\, f_i)\}_{i\in I}$ satisfying the condition that
for every $i,\, j\, \in\, I\times I$ with $U_i\bigcap U_j\, \not=\, \emptyset$, and every connected
component $V_c\, \subset\, U_i\bigcap U_j$, there is an element
$A^c_{j,i} \, \in\, \text{Aut}({\mathbb C}{\mathbb P}^1)\,=\, \text{PGL}(2, {\mathbb C})$ such that the
map $(f_j\circ f^{-1}_i)\big\vert_{f_i(V_c)}$ is the restriction of the automorphism $A^c_{j,i}$ of
${\mathbb C}{\mathbb P}^1$ to the open subset $f_i(V_c)$.
Two holomorphic coordinate atlases $\{(U_i,\, f_i)\}_{i\in I}$ and $\{(U_i,\, f_i)\}_{i\in I'}$
satisfying the above condition are called {\it equivalent} if their union
$\{(U_i,\, f_i)\}_{i\in I\cup
I'}$ also satisfies the above condition. A {\it projective structure} on $X$ is an equivalence class of
holomorphic coordinate atlases satisfying the above condition (see \cite{Gu}).

Take the extension $E$
$$
0\, \longrightarrow\, {\mathcal O}_X \, \longrightarrow\, E \, \longrightarrow\, TX \, \longrightarrow\, 0
$$
corresponding to $1\, \in\, H^1(X,\, K_X)\,=\, \mathbb C$. Note that there is exactly one nontrivial extension of 
$TX$ by ${\mathcal O}_X$ up to the scalings of ${\mathcal O}_X$. Giving a projective structure on $X$ is equivalent 
to giving a holomorphic connection on the projective bundle ${\mathbb P}(E)$. More precisely, projective structures 
on $X$ are identified with the quotient of the space of all holomorphic connections on ${\mathbb P}(E)$ by the group 
of all holomorphic automorphisms of ${\mathbb P}(E)$ \cite{BR}, \cite{Gu}. From this it follows that the space of 
all projective structures on $X$ is an affine space modeled on $H^0(X,\, K^{\otimes 2}_X)\,=\, (\Omega^1_{{\mathcal 
M}_g})_X$; see \cite{Gu}, \cite{BR}. Let ${\mathcal P}_g$ denote the space of all pairs $(X,\, \rho)$, where $X\, 
\in\, {\mathcal M}_g$ and $\rho$ is a projective structure on $X$. From the above description of projective 
structures on $X$ in terms of the holomorphic connections on ${\mathbb P}(E)$ it follows that ${\mathcal P}_g$ has a 
natural structure of a Deligne--Mumford stack. Let
\begin{equation}\label{e7}
\psi\, :\, {\mathcal P}_g \, \longrightarrow\, {\mathcal M}_g 
\end{equation}
be the natural projection. We note that ${\mathcal P}_g$ is a holomorphic torsor over ${\mathcal M}_g$
for the cotangent bundle $\Omega^1_{{\mathcal M}_g}$.

Every Riemann surface admits a projective structure. In fact, the uniformization theorem produces a projective 
structure, because the automorphism groups of $\mathbb C$, ${\mathbb C}{\mathbb P}^1$ and the upper-half plane 
$\mathbb H$ are all contained in $\text{PGL}(2, {\mathbb C})$. Consequently, the uniformization theorem produces a 
$C^\infty$ section
\begin{equation}\label{e8}
\Psi\, :\, {\mathcal M}_g \, \longrightarrow\, {\mathcal P}_g
\end{equation}
of the projection $\psi$ in \eqref{e7}. (See \cite{BCFP} for another canonical section of ${\mathcal P}_g$.)

\section{Holomorphic isomorphism of torsors}

We will construct a new holomorphic $\Omega^1_{{\mathcal M}_g}$--torsor from ${\rm Conn}({\mathcal L})$ in
\eqref{e4} by simply scaling the action of $\Omega^1_{{\mathcal M}_g}$, while keeping the
holomorphic fiber bundle unchanged. Define
\begin{equation}\label{f2}
A^t\, :\, {\rm Conn}({\mathcal L})\times_{{\mathcal M}_g} \Omega^1_{{\mathcal M}_g}\, \longrightarrow\,
{\rm Conn}({\mathcal L})\, ,\ \ (z,\, v)\, \longmapsto\, A\left(z,\, \frac{\sqrt{-1}}{6\pi}\cdot v\right)\, ,
\end{equation}
where $A$ is the map in \eqref{eA}; the map $A^t$ is
holomorphic because $A$ is so. The resulting holomorphic $\Omega^1_{{\mathcal M}_g}$--torsor
$({\rm Conn}({\mathcal L}),\, A^t)$ will be denoted by ${\rm Conn}^t({\mathcal L})$. This
${\rm Conn}^t({\mathcal L})$ can be interpreted as the bundle of connections on the (nonexistent)
line bundle ${\mathcal L}^{\otimes \frac{\sqrt{-1}}{6\pi}}$.

The $C^\infty$ section of ${\rm Conn}^t({\mathcal L})$ given by the section $\Phi$ (in \eqref{e6}) of ${\rm 
Conn}({\mathcal L})$ will also be denoted by $\Phi$. Since the two holomorphic fiber bundles
${\rm Conn}^t({\mathcal L})$ and ${\rm Conn}({\mathcal L})$ coincide, this should not cause any confusion.
For the same reason the projection of ${\rm Conn}^t({\mathcal L})$ to ${\mathcal M}_g$ will be denoted by
$\phi$ (as in \eqref{e4}).

A $C^\infty$ \textit{isomorphism} between
the $\Omega^1_{{\mathcal M}_g}$--torsors ${\rm Conn}^t({\mathcal L})$ and ${\mathcal P}_g$ (constructed
in \eqref{e7}) is a diffeomorphism
$$
F\, :\, {\rm Conn}^t({\mathcal L})\, \longrightarrow\, {\mathcal P}_g
$$
such that
\begin{enumerate}
\item $\psi\circ F\,=\, \phi$, where $\psi$ is the projection in \eqref{e7}, and

\item $F(c+w)\,=\, F(c)+ w$, for all $c\, \in\, {\rm Conn}^t({\mathcal L})_X$,
$w\, \in\, (\Omega^1_{{\mathcal M}_g})_X$ and $X\, \in\, {\mathcal M}_g$.
\end{enumerate}

A \textit{holomorphic isomorphism} between the $\Omega^1_{{\mathcal M}_g}$--torsors
${\rm Conn}^t({\mathcal L})$ and ${\mathcal P}_g$ is a $C^\infty$ isomorphism $F$ as above satisfying the
condition that $F$ is a biholomorphism.

\begin{theorem}\label{thm1}\mbox{}
\begin{enumerate}
\item There is exactly one holomorphic isomorphism between the $\Omega^1_{{\mathcal M}_g}$--torsors
${\rm Conn}^t({\mathcal L})$ and ${\mathcal P}_g$.

\item The holomorphic isomorphism between the $\Omega^1_{{\mathcal M}_g}$--torsors
${\rm Conn}^t({\mathcal L})$ and ${\mathcal P}_g$ takes the section $\Phi$ in \eqref{e6} to the section
$\Psi$ in \eqref{e8}.
\end{enumerate}
\end{theorem}

\begin{proof}
We will first prove that there is at most one holomorphic isomorphism between the
two $\Omega^1_{{\mathcal M}_g}$--torsors ${\rm Conn}^t({\mathcal L})$ and ${\mathcal P}_g$. To prove this,
for $i\, =\, 1,\, 2$, let
$$
F_i\, :\, {\rm Conn}^t({\mathcal L})\, \longrightarrow\, {\mathcal P}_g
$$
be a holomorphic isomorphism. Consider the difference
$$
F_1-F_2\, :\, {\rm Conn}^t({\mathcal L})\, \longrightarrow\, \Omega^1_{{\mathcal M}_g}\, ,\ \ c\, \longmapsto\,
F_1(c)-F_2(c)
$$
defined using the $\Omega^1_{{\mathcal M}_g}$--torsor structure on ${\mathcal P}_g$. Since
$F_i(c+w)\,=\, F_i(c)+ w$ for all $c\, \in\, {\rm Conn}^t({\mathcal L})_X$,
$w\, \in\, (\Omega^1_{{\mathcal M}_g})_X$ and $X\, \in\, {\mathcal M}_g$, we conclude that
$$(F_1-F_2)(c+w)\,=\, (F_1-F_2)(c)\, .$$ Consequently, $F_1-F_2$ descends to a holomorphic $1$-form on
${\mathcal M}_g$. But there is no nonzero holomorphic $1$-form on ${\mathcal M}_g$
\cite[Theorem 3.1]{FPT}; recall that $g\, \geq\, 5$. This implies that $F_1\,=\, F_2$. In other words,
there is at most one holomorphic isomorphism between the
two $\Omega^1_{{\mathcal M}_g}$--torsors ${\rm Conn}^t({\mathcal L})$ and ${\mathcal P}_g$.

We will now construct a $C^\infty$ isomorphism $\mathbb F$ between the
two $\Omega^1_{{\mathcal M}_g}$--torsors ${\rm Conn}^t({\mathcal L})$ and ${\mathcal P}_g$.
Take any $X\, \in\, {\mathcal M}_g$ and any $c\, \in\, \phi^{-1}(X)\,\subset\, {\rm Conn}^t({\mathcal L})
\,=\, {\rm Conn}({\mathcal L})$, where $\phi$ as before is the projection of ${\rm Conn}^t({\mathcal L})$
to ${\mathcal M}_g$. So $c\,=\, \Phi(X)+w$, where $\Phi$ is the section in
\eqref{e6} and $w\, \in\, (\Omega^1_{{\mathcal M}_g})_X$. Now define
$$
{\mathbb F}(c)\,=\, \Psi(X)+ w\, ,
$$
where $\Psi$ is the section in \eqref{e8}. This produces a map
\begin{equation}\label{e9}
{\mathbb F}\, :\, {\rm Conn}^t({\mathcal L})\, \longrightarrow\, {\mathcal P}_g\, .
\end{equation}
It is straightforward to check that this map ${\mathbb F}$ is a $C^\infty$ isomorphism between the
two $\Omega^1_{{\mathcal M}_g}$--torsors ${\rm Conn}^t({\mathcal L})$ and ${\mathcal P}_g$.

It is evident that ${\mathbb F}(\Phi({\mathcal M}_g))\,=\, \Psi({\mathcal M}_g)$.

To complete the proof we need to show that $\mathbb F$ is a biholomorphism.

The real tangent bundles of ${\rm Conn}^t({\mathcal L})$ and ${\mathcal P}_g$ will be denoted by
$T^{\mathbb R}{\rm Conn}^t({\mathcal L})$ and $T^{\mathbb R}{\mathcal P}_g$ respectively. Let
$$
d{\mathbb F}\, :\, T^{\mathbb R}{\rm Conn}^t({\mathcal L})\, \longrightarrow\,{\mathbb F}^*T^{\mathbb R}{\mathcal P}_g
$$
be the differential of the map ${\mathbb F}$ in \eqref{e9}.
Let $J_C$ (respectively, $J_P$) be the almost complex structure on ${\rm Conn}^t({\mathcal L})$
(respectively, ${\mathcal P}_g$). Since ${\mathbb F}$ is a diffeomorphism, to prove that
$\mathbb F$ is a biholomorphism it is enough to show that
\begin{equation}\label{e10}
d{\mathbb F}\circ J_C\,=\, ({\mathbb F}^*J_P)\circ (d{\mathbb F})
\end{equation}
as maps from $T^{\mathbb R}{\rm Conn}^t({\mathcal L})$ to ${\mathbb F}^*T^{\mathbb R}{\mathcal P}_g$; the
automorphism of ${\mathbb F}^*T^{\mathbb R}{\mathcal P}_g$ given by the automorphism $J_P$ of
$T^{\mathbb R}{\mathcal P}_g$ is denoted by ${\mathbb F}^*J_P$.

Take any point $X\, \in\, {\mathcal M}_g$. The restriction of $\mathbb F$ to $\phi^{-1}(X)\, \subset\, {\rm 
Conn}^t({\mathcal L})$ is a biholomorphism with $\psi^{-1}(X)$, where $\psi$ is the map in \eqref{e7}. More 
precisely, this restriction is an isomorphism of affine spaces modeled on the vector space $(\Omega^1_{{\mathcal 
M}_g})_X$. Therefore, the equation in \eqref{e10} holds for the subbundle of $T^{\mathbb R}{\rm Conn}^t({\mathcal 
L})$ given by the relative tangent bundle for the projection $\phi$ to ${\mathcal M}_g$.

For convenience, the image in ${\rm Conn}^t({\mathcal L})$ of the map $\Phi$ (see \eqref{e6})
will be denoted by $Y$. Let
\begin{equation}\label{e11}
\iota\, :\, Y\, \hookrightarrow\, {\rm Conn}^t({\mathcal L})
\end{equation}
be the inclusion map. Since $\Phi$ is just a $C^\infty$ section, this $Y$ does not inherit any
complex structure from ${\rm Conn}^t({\mathcal L})$. Note that $Y$ can be given a complex structure, because
the restriction of the projection $\phi$ (see \eqref{e4}) to $Y$ is a diffeomorphism of $Y$ with
${\mathcal M}_g$, so the complex structure on ${\mathcal M}_g$ produces a complex structure on $Y$. It should
be clarified that for this complex structure on $Y$ the inclusion map $\iota$ in \eqref{e11} is not holomorphic,
because the section $\Phi$ is not holomorphic.

Using the differential $d\iota\, :\, T^{\mathbb R}Y \, \longrightarrow\,\iota^* T^{\mathbb R}{\rm Conn}^t({\mathcal L})$
of the embedding $\iota$ in \eqref{e11} the tangent bundle $T^{\mathbb R} Y$ is realized as a $C^\infty$ subbundle
of $\iota^* T^{\mathbb R}{\rm Conn}^t({\mathcal L})$. So we have
\begin{equation}\label{e12}
T^{\mathbb R} Y\, \subset\, \iota^* T^{\mathbb R}{\rm Conn}^t({\mathcal L})\, \subset\,
T^{\mathbb R}{\rm Conn}({\mathcal L})\, .
\end{equation}

To prove \eqref{e10}, take any point $\gamma\, \in\, {\rm Conn}^t({\mathcal L})$, and any tangent vector
\begin{equation}\label{ev}
v\, \in\, T^{\mathbb R}_\gamma {\rm Conn}^t({\mathcal L})
\end{equation}
at $\gamma$. We noted earlier that \eqref{e10} holds for the relative tangent bundle for the projection
$\phi$ to ${\mathcal M}_g$. So we assume that $v$ is not vertical for the projection $\phi$.

Denote $\phi(\gamma)\, \in\, {\mathcal M}_g$ by $z$,
and also denote $\Phi(z)\, \in\, {\rm Conn}^t({\mathcal L})$ by $\delta$. So we have
\begin{equation}\label{e13}
w_0 \, :=\, \gamma-\delta \, \in\, (\Omega^1_{{\mathcal M}_g})_z
\end{equation}
using the $\Omega^1_{{\mathcal M}_g}$--torsor structure of ${\rm Conn}^t({\mathcal L})$ (see \eqref{f2}). Let
${\widetilde v}\, \in\, T^{\mathbb R}_z {\mathcal M}_g$ be the image of $v$ in \eqref{ev} by the differential
$d\phi$ of $\phi$. Let
\begin{equation}\label{e14}
u\, \in\, T^{\mathbb R}_\delta Y
\end{equation}
be the image of ${\widetilde v}$ by the differential $d\Phi$ of $\Phi$.

Take a holomorphic $1$-form $w$ defined on some analytic neighborhood $U$ of $z\, \in\, {\mathcal M}_g$.
Then $w$ defines a biholomorphism
\begin{equation}\label{tw}
{\mathcal T}_w\, :\, \phi^{-1}(U)\, \longrightarrow\, \phi^{-1}(U)\, ,\ \ \alpha\,\longmapsto\, \alpha+
w(\phi(\alpha))\, ;
\end{equation}
here the $\Omega^1_{{\mathcal M}_g}$--torsor structure of ${\rm Conn}^t({\mathcal L})$ is used.
Now choose the $1$-form $w$ such that
\begin{itemize}
\item $w(z)\,=\, w_0$; see \eqref{e13} (this condition is clearly equivalent to the condition that
${\mathcal T}_w(\delta) \,=\, \gamma$), and 

\item the differential $d{\mathcal T}_w$ of ${\mathcal T}_w$ takes $u$ in \eqref{e14} to $v$ (see \eqref{ev}).
\end{itemize}
Note that since $v$ is not vertical for the projection $\phi$, such a $1$-form $w$ exists.
We have the following biholomorphism of $\psi^{-1}(U)\, \subset\, {\mathcal P}_g$, where $\psi$ is the
projection in \eqref{e7}:
$$
{\mathcal T}^1_w\, :\, \psi^{-1}(U)\, \longrightarrow\, \psi^{-1}(U)\, ,\ \ \alpha\,\longmapsto\, \alpha+
w(\psi(\alpha)).
$$
{}From the construction of ${\mathbb F}$ in \eqref{e9} it follows immediately that
\begin{equation}\label{e15}
{\mathbb F}\circ {\mathcal T}_w\,=\, {\mathcal T}^1_w\circ {\mathbb F}
\end{equation}
as maps from $\phi^{-1}(U)\, \subset\, {\rm Conn}^t({\mathcal L})$ to $\psi^{-1}(U)
\, \subset\, {\mathcal P}_g$.

Since ${\mathcal T}_w$ (respectively, ${\mathcal T}^1_w$)
is a biholomorphism, its differential $d{\mathcal T}_w$ (respectively, $d{\mathcal T}^1_w$) preserves
the almost complex structure $J_C$ (respectively, $J_P$) on $\phi^{-1}(U)$ (respectively, $\psi^{-1}(U)$).
Therefore, from \eqref{e15} we conclude the following:

Take any point $\mu\, \in\, \phi^{-1}(U)\, \subset\,
{\rm Conn}^t({\mathcal L})$. Then \eqref{e10} holds for all vectors in the tangent space
$T^{\mathbb R}_\mu {\rm Conn}^t({\mathcal L})$ if and only if \eqref{e10} holds for all vectors in the tangent space
$T^{\mathbb R}_{\mu+w(\phi (\mu))}{\rm Conn}^t({\mathcal L})$, where $w$ is the
above holomorphic $1$-form. More precisely, \eqref{e10} holds for a tangent vector
$v_0\, \in\, T^{\mathbb R}_\mu {\rm Conn}^t({\mathcal L})$ if and only if \eqref{e10} holds for 
$$(d{\mathcal T}_w)(v_0)\,\in\, T^{\mathbb R}_{\mu+w(\phi (\mu))}{\rm Conn}^t({\mathcal L})\, ,$$ where
$d{\mathcal T}_w\, :\, T^{\mathbb R}\phi^{-1}(U)\, \longrightarrow\, T^{\mathbb R}\phi^{-1}(U)$ is the
differential of the map ${\mathcal T}_w$ in \eqref{tw}.

Setting $\mu\,=\, \delta$ and $v_0\,=\, u$ (see \eqref{e13} and \eqref{e14}) in the above statement we obtain that
\eqref{e10} holds for $u\, \in\,
T^{\mathbb R}_\delta Y\, \subset\,
T_\delta {\rm Conn}^t({\mathcal L})$ if and only if \eqref{e10} holds for
$v\, \in\, T_\gamma {\rm Conn}^t({\mathcal L})$ in \eqref{ev}.

Consequently, to prove \eqref{e10} it suffices to establish it for all tangent vectors in the subspace
$T^{\mathbb R}Y$ in \eqref{e12}.

Let $q\, :\, {\mathcal V}\, \longrightarrow\, {\mathcal M}_g$ be a holomorphic $\Omega^1_{{\mathcal M}_g}$--torsor
on ${\mathcal M}_g$, and let $S\, :\, {\mathcal M}_g\, \longrightarrow\, {\mathcal V}$ be
a $C^\infty$ section of ${\mathcal V}$. From these we will construct a $C^\infty$ $(1,\,1)$-form
on ${\mathcal M}_g$. The almost complex structures on $\mathcal V$ and ${\mathcal M}_g$
will be denoted by $J_V$ and $J_M$ respectively. Let $$dS\, :\, T^{\mathbb R}{\mathcal M}_g
\, \longrightarrow\, T^{\mathbb R} {\mathcal V}$$ be the differential of the map $S$.
Take any $X\, \in\, {\mathcal M}_g$ and any $v\, \in\, T^{\mathbb R}_X{\mathcal M}_g$. Define
$$
\widetilde{S}(v)\, :=\, dS(J_M(v))- (J_V\circ dS)(v)\,\in\, T^{\mathbb R}_{S(X)} {\mathcal V}\ .
$$
Since $q$ is holomorphic, and $q\circ S\,=\, {\rm Id}_{{\mathcal M}_g}$, it can be shown that
the tangent vector $\widetilde{S}(v)$ is vertical for the projection $q$. Indeed, we have
$$
dq(dS(J_M(v))- (J_V\circ dS)(v))\,=\, dq(dS(J_M(v))) - dq((J_V\circ dS)(v))
$$
$$
=\, d{\rm Id}_{{\mathcal M}_g}(J_M(v)) - J_M(dq((dS)(v)))\,=\, J_M(v) -J_M(v) \,=\, 0\, .
$$
So $\widetilde{S}(v)$ is vertical for the projection $q$. On the
other hand, using the $\Omega^1_{{\mathcal M}_g}$--torsor structure of $\mathcal V$, the
vertical tangent space at $S(X)\, \in\, {\mathcal V}$ is identified with $(\Omega^1_{{\mathcal M}_g})_X$. Also,
$T^{\mathbb R}_X{\mathcal M}_g$ is identified with the real vector
space underlying $(\Omega^{0,1}_{{\mathcal M}_g})_X\,=\,
(\overline{\Omega}^1_{{\mathcal M}_g})_X$. Using these we have $$\widetilde{S}\, \in\,
C^\infty({\mathcal M}_g,\, \Omega^{1,1}_{{\mathcal M}_g})$$ \cite{Iv}.
Note that $S$ is a holomorphic section if and only if $\widetilde{S}\,=\, 0$. This form
$\widetilde{S}$ is called the {\it obstruction} for $S$ to be holomorphic (see \cite{Iv}).

The obstruction for the section $\Psi$ in \eqref{e8} to be holomorphic is the Weil--Petersson form
$\omega_{WP}$ in \eqref{e0} \cite[p.~214, Theorem 1.7]{Iv}, \cite{ZT2}. On the other hand, the obstruction for the
section $\Phi$ of ${\rm Conn}({\mathcal L})$ in \eqref{e6} to be holomorphic is the $(1,\, 1)$-component of
the curvature of the connection on
$\mathcal L$ corresponding to $\Phi$. From \eqref{e5} we know that this curvature itself is of type $(1,\, 1)$
and it is $\frac{\sqrt{-1}}{6\pi}\omega_{WP}$. Now from \eqref{f2} we conclude that the
the obstruction for the section $\Phi$ of ${\rm Conn}^t({\mathcal L})$ to be holomorphic is the
Weil--Petersson form $\omega_{WP}$. Comparing the obstructions for the
sections $\Phi$ and $\Psi$ we conclude that \eqref{e10} holds for all tangent vectors in the subspace
$T^{\mathbb R}Y$ in \eqref{e12}, because the two obstructions coincide. This completes the proof.
\end{proof}

\section*{Acknowledgements}

We would like to thank Alessandro Ghigi for helpful remarks and discussions on the topic of the article. The second 
and third named authors are partially supported by INdAM - GNSAGA. The second named author is partially supported 
by ``2017-ATE-0253'' (Dipartimento di Matematica e Applicazioni - Universit\`a degli Studi di Milano-Bicocca). The 
third named author is partially supported by PRIN 2017\emph{``Moduli spaces and Lie theory''} and by (MIUR): 
Dipartimenti di Eccellenza Program (2018-2022) - Dept. of Math. Univ. of Pavia.
The first named author is partially supported by a J. C. Bose Fellowship.


\begin{thebibliography}{ZZZZZ}

\bibitem[AC]{AC} E. Arbarello and M. Cornalba, The Picard groups of the moduli spaces of curves,
{\it Topology} {\bf 26} (1987), 153--171. 

\bibitem[At]{At} M. F. Atiyah, Complex analytic connections in fibre
bundles, \textit{Trans. Amer. Math. Soc.} \textbf{85} (1957), 181--207.

\bibitem[BGS]{BGS} J.-M. Bismut, H. Gillet and C. Soul\'e, Analytic torsion and holomorphic determinant
bundles. I. Bott-Chern forms and analytic torsion, {\it Comm. Math. Phys.} {\bf 115} (1988), 49--78.

\bibitem[BHS]{BHS} I. Biswas, J. Hurtubise and J. Stasheff, A construction of a universal connection,
{\it Forum Math.} {\bf 24} (2012), 365--378.

\bibitem[BR]{BR} I. Biswas and A. K. Raina, Projective structures on a Riemann surface, II, {\it 
Inter. Math. Res. Not.} (1999), no. 13, 685--716.

\bibitem[BCFP]{BCFP} I. Biswas, E. Colombo, P. Frediani and G. P. Pirola, A Hodge theoretic projective
structure on compact Riemann surfaces, {\it Jour. Math. Pures Appl.} {\bf 149} (2021), 1--27.

\bibitem[DM]{DM} P. Deligne and D. Mumford, The irreducibility of the space of curves of given genus,
{\it Inst. Hautes \'Etudes Sci. Publ. Math.} {\bf 36} (1969), 75--109. 

\bibitem[FPT]{FPT} F. F. Favale, G. P. Pirola and S. Torelli, Holomorphic $1$-forms on the moduli
space of curves, arXiv:2009.10490.

\bibitem[Gu]{Gu} R. C. Gunning, {\it On uniformization of complex manifolds: the
role of connections}, Princeton Univ. Press, 1978.

\bibitem[Ha]{Ha} J. Harer, The second homology group of the mapping class group of an orientable
surface, {\it Invent. Math.} {\bf 72} (1983), 221--239. 

\bibitem[Iv]{Iv} N. V. Ivanov, Projective structures, flat bundles and K\"ahler metrics on moduli 
spaces, {\it Math. USSR-Sb.} {\bf 61} (1988), 211--224.

\bibitem[Ko]{Ko} S.~Kobayashi, {\em Differential geometry of complex vector bundles}, Publications of the 
Mathematical Society of Japan, 15. Kan\^o Memorial Lectures, 5. Princeton University Press,
Princeton, NJ; Princeton University Press, Princeton, NJ, 1987.

\bibitem[Mu]{Mu} D. Mumford, Abelian quotients of the Teichm\"uller modular group,
{\it Jour. Analyse Math.} {\bf 18} (1967), 227--244.

\bibitem[Qu]{Qu} D. G. Quillen, Determinants of Cauchy-Riemann operators on Riemann surfaces.
{\it Funct. Anal. Appl.} {\bf 19} (1985), 31--34.

\bibitem[ZT1]{ZT1} P. G. Zograf and L. A. Takhtadzhyan, A local index theorem for families of
$\overline\partial$-operators on Riemann surfaces, {\it Russian Math. Surveys} {\bf 42-6} (1987), 169--190.

\bibitem[ZT2]{ZT2} P. G. Zograf and L. A. Takhtadzhyan, On the uniformization of Riemann surfaces and on 
the Weil-Petersson metric on the Teichm\"uller and Schottky spaces,
{\it Math. USSR-Sb.} {\bf 60} (1988), 297--313.

\end{thebibliography}
\end{document}